\documentclass[12pt,reqno]{amsart}
\usepackage{microtype}
\usepackage{a4wide}
\usepackage{preamble}
\usepackage{subfiles}
\DeclareMathOperator{\Span}{span}
\newcommand{\modt}{{\rm mod}(2)}
\newcommand{\Ann}{\mathcal A}

\newcommand{\sS}{\mathcal S}
\newcommand{\sT}{\mathcal T}
\newcommand{\sP}{\mathcal P}

\newcommand{\bracket}[1]{\llbracket #1\rrbracket}

\title{$L^2$ Curvature Bounds for  Area-Minimizing Currents Mod(2)}
\author[Lihn]{Zachary Lihn}
\address{Department of Mathematics, Princeton University, Princeton, NJ, USA}
\email{zl3140@princeton.edu}

\begin{document}
\begin{abstract}
    We prove \emph{a priori} interior $L^2$ bounds for the second fundamental form of area-minimizing currents modulo $2$  in every dimension $n\geq 2$ and arbitrary codimension. The bounds depend only on the dimension, codimension, and an upper bound for the mass.
    As a corollary, we show that area-minimizing surfaces modulo $2$ with bounded area  lie in finitely many diffeomorphism classes. This answers a question of \cite{Liu2020WeakCounterexamples}.
\end{abstract}
\maketitle

\section{Introduction}

In this paper we are interested in the regularity of $\modt$ area-minimizing  currents of dimension $n\geq 2$ and arbitrary codimension $m$ in the Euclidean space $\R^{m+n}$. 
Our main result is that the $L^2$ norm of the second fundamental form is \emph{a priori} bounded in terms of the mass.

\begin{thm}\label{thm:main-curvature}
    Let $n\geq 2, m\geq 1$ be integers and let $\Lam>0$. 
    Suppose $M$ is an $n$-dimensional area-minimizing current $\modt$ in $B_8\subset \R^{n+m}$ with $\ptl M\mres B_8 =0$ and $|M|(B_8) \leq \Lambda$. Then there exists a constant $C=C(m,n,\Lambda)>0$ such that 
    \begin{equation}\label{eq:main-L2-bound}
        \int_{B_1} |A_M|^2 d|M| \leq C(n,m,\Lambda).
    \end{equation}
    Here $A_M$ is the second fundamental form of $\reg(M)$ extended by zero across the singular set.
\end{thm}

As a corollary, we prove that $\modt$ area-minimizing surfaces in arbitrary codimension and bounded area lie in finitely many diffeomorphism classes.
    \begin{cor} \label{cor:2d-topology-bound}
    Let $m\geq 1$ and  $\Lam>0$. 
    There exists $J=J(m,\Lam)$ and smooth surfaces $S_1,\ldots, S_J$ such that if $M$ is any two-dimensional area-minimizing current $\modt$ in $B_{16}$ with $\ptl M\mres B_{16} =0$, $|M|(B_{16})\leq \Lam$, then 
    $\#(\sing (M)\cap B_1) \leq J$ and there is a radius $r\in (1,2)$ such that $\reg(M)\cap B_r$ is diffeomorphic to $S_j$ for some $j\in \{1,\ldots, J\}$.
\end{cor}

The interest in curvature estimates  stems from their importance in the regularity and existence theory for minimal submanifolds. 
In the case of minimal hypersurfaces, the curvature estimates of \cite{Simons1968,Simon1976Curvature,SchoenSimonYau1975, SchoenSimon1981,Bellettini2025} have proven enormously successful with powerful applications such as the min-max existence theory for minimal hypersurfaces in Riemannian manifolds (see for instance  \cite{Pitts1981,ColdingDeLellis2003,Smith1982,MarquesNeves2014}). 

On the other hand, curvature estimates for the second fundamental form  of minimal submanifolds in higher codimension are much more subtle. 
While the rectifiability and partial regularity for the singular set of area-minimizing currents in higher codimension is now known \cite{Simon1993,DeLellisSpadaro2014,DeLellisSpadaro2016III,DeLellisSkorobogatova2025I,DeLellisSkorobogatova2025II,DeLellisMinterSkorobogatova2023III,KrummelWickramasekera2023I,KrummelWickramasekera2023II,KrummelWickramasekera2026III},
the examples of \cite{Liu2020WeakCounterexamples} show that the  estimates of \cite{SchoenSimonYau1975} in fact fail for area-minimizing integral currents. Indeed, in \cite{Liu2020WeakCounterexamples} a sequence of smooth  area-minimizing integral currents is constructed which have  uniformly bounded area and genus going to infinity (and hence $L^2$ norm of $A_M$ diverging as well).
Nevertheless, the examples of \cite{Liu2020WeakCounterexamples} are area-minimizing for integral coefficients and one can hope that the case of $\modt$ coefficients satisfies better properties.
Indeed, in those examples the sequence converges to a double copy of a plane and so cannot be $\modt$ area-minimizing.

By the work of \cite{Naber2020}, it is relatively straightforward to show that $A_M$ lies in the Lorentz space $L^{2,\infty}$ by using Allard's regularity theorem \cite{Allard1972}.
However, the methods there do not extend fully to the endpoint estimate $A_M\in L^{2}$ and require new ideas. 
Indeed, the main obstacle to the $L^2$ estimate is the need to sum across an unbounded number of scales in annular regions (see Section \ref{subsec:annular-decomposition}). Our main contribution in Theorem \ref{thm:annular-curvature} is to show that $A_M$ has uniformly bounded (in fact, uniformly small) $L^2$ norm across the entire annular region.
Combined with the Annular Decomposition Theorem~\ref{thm:annular-decomposition}, a simple covering argument gives Theorem~\ref{thm:main-curvature}. For the reader's convenience, the necessary changes to \cite{ChowJiangNaber} to establish the annular decomposition and structure theorems for stationary varifolds are sketched in Appendix \ref{sec:A:annular}.
We give a more detailed sketch of the proof below.

We note that  Theorem~\ref{thm:main-curvature} and Corollary~\ref{cor:2d-topology-bound}  both likely generalize in a straightforward manner to Riemannian manifolds with bounded geometry. The proof carries over directly with additional lower-order errors stemming from  ambient curvature. However, to keep the presentation simple we restrict ourselves to Euclidean space in this paper.

\subsection{Sketch of the Proof}
The proof of Theorem~\ref{thm:main-curvature} follows the model of \cite{NaberValtorta2019YangMills,Cheeger2015,JiangNaber2021,NaberValtorta2024HarmonicMaps} for proving scale-critical curvature estimates  in other geometric equations. 
We begin by decomposing the regular part of $M$ into two types of regions, which we call the $(n-1,\eps)$-symmetric regions and the $\delta$-annular regions. By the Annular Decomposition Theorem~\ref{thm:annular-decomposition}, we can construct such a decomposition with a uniformly controlled number of annular and symmetric regions. 
It therefore suffices to estimate the second fundamental form of each type of region independently.
The $\modt$ structure ensures that the analysis of the $(n-1,\eps)$-symmetric regions is straightforward. Indeed, by a contradiction and compactness argument along with Allard's regularity theorem it is immediate to see that these regions are diffeomorphic to (a ball in) an $n$-dimensional flat plane and have uniformly controlled second fundamental form by standard elliptic regularity theory (see Proposition~\ref{prop:symmetry-eps-regularity}). 
Thus the bulk of our analysis will focus on the annular regions, which is where the novelty of this paper lies.

The main estimate on annular regions is Theorem~\ref{thm:annular-curvature}, which, roughly speaking, says that for a $\delta$-annular region $\cA\subset B_8$ we have 
\begin{equation}
    \int_{\cA\cap B_1} |A_M|^2 d|M| \leq C (\eps + \delta)
\end{equation}
for any $\eps>0$ and $\delta<\delta(n,m,\Lam,\eps)$. 
Let us first explain the proof for the case of surfaces when $n=2$, since this already suffices to illustrate  the central difficulties and ideas. 
In this case the $\delta$-annular region $\Ann = B_{16}\setminus \ol B_{r_0}$ is merely an ordinary annulus, with $r_0\in [0,8)$  an unknown inner radius. 
From the definition of an annular region, the classification of area-minimizing cones $\modt$ \cite{Morgan1982}, and Allard's regularity theorem \cite{Allard1972},  
it is easy to see that $M\mres (B_{4r}\setminus \ol B_{r/4})$ is smooth, diffeomorphic to a region of a 2-dimensional $\modt$ area-minimizing cone (away from the spine), and has the estimate 
\begin{equation} \label{eq:intro-weak-sff}
    |A_M|^2 \leq C \eps r^{-2}
\end{equation} 
in this region. 
The key issue, however, is that the estimate \ref{eq:intro-weak-sff} is not summable over all scales, and in particular diverges logarithmically in the bottom scale $r_0$ which we cannot control (indeed, at a singular point $r_0=0$). This is the reason for the estimate $A_M\in L^{2,\infty}$ but not the full estimate $A_M\in L^2$.

Our  solution is to estimate the second fundamental form in terms of the mass density drop (see also \cite{JiangNaber2021} for a similar argument for Einstein manifolds). 
If we let $\Theta_r = \frac{|M|(B_r)}{\ome_n r^n}$ be the usual mass density for $M$ centered at the origin, we will prove an estimate of the form
\begin{equation} \label{eq:intro-density-drop}
    \int_{B_{2r}\setminus B_r} |A_M|^2 d|M| \leq C(m,n,\Lam) |\Theta_{4r}-\Theta_{r/4}|
\end{equation}
for all $r\in [r_0,2)$.
Since $\Theta_r$  is monotone, we may telescope \ref{eq:intro-density-drop} across all dyadic scales in $[r_0,2)$ to obtain
\begin{equation*}
    \int_{B_2\setminus B_{4r_0}} |A_M|^2 d|M| \leq C |\Theta_8 - \Theta_{r_0}|< C \delta,
\end{equation*}
where the second inequality is immediate from the definition of an annular region. 
For technical reasons we must estimate the region $B_{4r_0}\setminus B_{r_0}$ separately, but this is no issue since this region occurs only at a bounded number of scales for which we may apply \ref{eq:intro-weak-sff}.

It remains to describe \ref{eq:intro-density-drop}. We apply minimality to estimate the full norm $|A_M|^2$ by the component in the radial direction $r^{-2}|A_M(\cdot, \pi_{T_xM}(x))|^2 = r^{-2} |\nabla^\perp_M \pi_{T_xM}^\perp (x)|^2$ by the Gauss-Weingarten formula. On the other hand the vector field $Z= \pi_{T_xM}^\perp (x)$ is a Jacobi field, and therefore satisfies the Jacobi equation $\Delta_M^\perp Z+ \td A_M(Z)=0$ (see Section~\ref{subsec:jacobi-fields} for the notation). Integrating by parts and using \ref{eq:intro-weak-sff} we obtain
\begin{equation*}
  r^{-2} \int_{B_{2r}\setminus B_r} |\nabla^\perp_M Z|^2 d|M|\leq C r^{-4} \int_{B_{4r}\setminus B_{r/4}} |\pi_{T_xM}^\perp (x)|^2 d|M| \leq C|\Theta_{4r}-\Theta_{r/4}|
\end{equation*}
by the monotonicity formula, which finishes the proof.

The proof in higher dimensions $n>2$ proceeds in much the same way except that we have additional components of the second fundamental form corresponding roughly to the  directions of approximate translation invariance in the annular region. 
Nevertheless we can still estimate the $L^2$ norm of $A_M$ by the density drop (integrated along the \emph{approximate singular set}), by taking
 again the Jacobi fields corresponding to translational and radial symmetry.  
 For technical reasons we will not work in annuli but instead in balls of radius comparable to the distance to the approximate singular set. We will also replace $\Theta_r$ by a mollified version $\theta_r(x)$ which is easier to work with (see Section~\ref{subsec:effective-symmetry}). Nevertheless, the ideas of the proof are similar.

\section{Preliminaries}
We collect here the facts about quantitative stratification which we need for the proof. 
The proofs are slight modifications of those in \cite{Naber2020,Cheeger2015,JiangNaber2021,NaberValtorta2019YangMills}, and especially the forthcoming manuscript \cite{ChowJiangNaber}, which we follow  closely. 
For the reader's convenience we provide a sketch of the necessary modifications for stationary integral varifolds in the appendices. 
For background on $\modt$ currents we refer to \cite{Fleming1966,Federer1970,Chodosh2012GMT} and \cite{Simon1993,Morgan1982} for their regularity theory.

\subsection{Effective Symmetry, Regularized Mass} \label{subsec:effective-symmetry}

\begin{defn}[Regularized Mass] \label{def:regularized-mass}
    Let $M$ be an integral stationary $n$-varifold in $B_4$. We define
    \begin{equation}
        \theta_r(x) :=\frac{1}{\ome_n r^n}\int_{B_r(x)}\left(1-\frac{|y-x|^2}{r^2} \right)d|M|(y).
    \end{equation}
\end{defn}

\begin{prop}[Regularized Monotonicity Formula] \label{prop:monotonicity}
    Let $M$ be an integral stationary $n$-varifold in $B_4$. Then, for every $B_r(x)\subset B_4$ we have 
    \begin{equation}
        r\frac{d}{dr}\theta_r(x) = \frac{2}{\ome_n r^{n+2}}\int_{B_r(x)} \left| \pi_{T_yM}^\perp(y-x)\right|^2d|M|(y) 
    \end{equation}
    In particular, we have for every $1\leq \lam<2$,
    \begin{equation} \label{eq:monotonicity-pinching}
        \theta_{2r}(x)-\theta_r(x) \geq \frac{C(n,\lam)}{ r^{n+2}} \int_{B_{\lam r}(x) }|\pi_{T_yM}^\perp(y-x)|^2d|M|(y)
    \end{equation}
\end{prop}
\begin{proof}
    The proof is a straightforward modification of the usual monotonicity formula by testing with the vector field $X_{x,r}(y) = \left( 1- \frac{|y-x|^2}{r^2}\right)_+ (y-x)$.
    We may then integrate over $s\in [\lam r, 2r]$, restrict the inner integral to $B_{\lam r}(x)$, and integrate in $s$ to obtain \ref{eq:monotonicity-pinching}.
\end{proof}

We next introduce the notion of approximate symmetry. This is different from the definition in \cite{Naber2020} and is the direct analogue of that in \cite{ChowJiangNaber}. 
Here and in the rest of the paper we will let $|A|$ denote the Hilbert-Schmidt norm of a matrix $A$, and $\|A\|$ its operator norm. 
\begin{defn}[$(k,\delta)$-symmetry]
    Let $M$ be an integral $n$-varifold in $B_2$ with $|M|(B_2)<\infty$. We say that $M$ is $(k,\delta)$-symmetric on the ball $B_r(x)\subset B_2$ if there exists a  $k$-dimensional vector space $L^k$ such that 
    \begin{equation}
        \frac{1}{\ome_n r^n}\int_{B_r(x)} \left(|\pi_{T_yM}^\perp|_{L^k}|^2 + r^{-2} |\pi_{T_yM}^\perp \pi_{L^\perp}(y-x)|^2 \right) d|M|(y) <\delta.
    \end{equation}
\end{defn}

As in the proof of the rigidity case for the usual monotonicity formula, one can verify that $(k,0)$-symmetry implies conicality and a spine of at least $k$ dimensions. 
We leave the proof to the reader as it closely follows  standard proofs of the monotonicity formula.
\begin{lem} \label{lem:(k,0)-symmetry-cone}
    Suppose $M$ is a stationary integral  $n$-varifold and $L^k$ is a $k$-linear subspace with 
    \begin{equation*}
        \frac{1}{\ome_n r^n}\int_{B_r(0^{n+m})} \left( |\pi_{T_yM}^\perp|_L|^2 + r^{-2} |\pi_{T_yM}^\perp\pi_{L^\perp}(y)|^2\right) d|M|(y) =0.
    \end{equation*}
    Then $M$ is radially homogeneous and  translation invariant along $L$ in $B_r(0)$. In particular, $\spine(M)\supset L$.
\end{lem}

The next proposition  converts enough $(n-1)$-symmetry into regularity. It follows from the non-existence of nonflat $(n-1)$-symmetric $\modt$ area-minimizing cones and the multiplicity 1 hypothesis to apply Allard's regularity theorem, and fails dramatically without the $\modt$ area-minimizing assumption.

\begin{prop}[Symmetry-based regularity] \label{prop:symmetry-eps-regularity}
    There exists $\eps_{reg}=\eps_{reg}(n,m,\Lambda)>0$ such that the following holds. Suppose $M$ is a $\modt$ $n$-dimensional area-minimizing current in $B_4$ with $\ptl M\mres B_4 =0$ and $|M|(B_4)\leq \Lambda$. If $M$ is $(n-1,\eps_{reg})$-symmetric in $B_4$, then  $M$ is smooth in $B_{63/64}$ and 
    \begin{equation}
        |A_M|\leq c(m,n,\Lambda) \quad \text{in }B_{63/64}
    \end{equation}
    for some $c(m,n,\Lambda)>0$.
\end{prop}
\begin{proof}
    The proof is by contradiction. 
    If no such $\eps_{reg}$ existed, we could find a sequence of $n$-dimensional $\modt$ area-minimizing currents $M_i$ which are $(n-1,\eps_i)$-symmetric with respect to $(n-1)$-planes $L_i$ in $B_4$ with $\eps_i\downarrow 0$. 
    Taking a subsequence if necessary and passing to the limit, we see that there exists a $\modt$ area-minimizing current $M$ in $B_4$ with $\ptl M \mres B_4=0$ such that $M_i\to M$ as currents and as varifolds. Moreover by compactness of the Grassmannian we may assume $L_i\to L$ for some $(n-1)$-plane $L$.
    Passing the $(n-1,\eps_i)$-symmetry to the limit shows that $M$ satisfies the hypotheses of Lemma~\ref{lem:(k,0)-symmetry-cone} and hence $M$ is a cone with spine containing $L$. Thus $M$ splits as a product $M = \bracket{L} \times M_0$ with $M_0$ a $1$-dimensional $\modt$ area-minimizing cone in $L^\perp$. 
    But any such $M_0$ must be a multiplicity 1 line, so that $M$ is  an $n$-plane with multiplicity 1. 
    Since $M_i\to M$ we may apply Allard's regularity theorem and elliptic regularity to reach a contradiction for large enough $i$.
\end{proof}

\subsection{Annular Decomposition}
\label{subsec:annular-decomposition}

We now introduce the notion of annular region, which is where most of the analysis will take place. 
We refer to the appendices for the outlines of the proofs.

Let us first recall the notion of points being quantitatively linearly independent.
\begin{defn}[Effective General Position]
    We say that a collection of points 
    $\{x_0,\ldots, x_{k}\}\subset B_r(x)$ is in $\tau$-general position if for each $1\leq j\leq k$ we have 
    \begin{equation*}
        d(x_j-x_0, \Span\{x_1-x_0,\ldots, x_{j-1}-x_0\})>\tau r.
    \end{equation*}
\end{defn}

The definition of annular region  captures the idea of a current being approximately conical at many scales around an $(n-2)$-dimensional subspace. 
We will only introduce the top-dimensional stratum here since it is most relevant to Theorem~\ref{thm:main-curvature}. In the appendices we will also introduce the definitions of annular regions around the lower-dimensional strata. 

\begin{defn}[$\delta$-annular region] \label{def:annular-region}
    Let $M$ be an area-minimizing current $\modt$ in $B_{4r_a}(x_a)$ with $\ptl M \mres B_{4r_a}(x_a)=0$ and $|M|(B_{4r_a}(x_a))\leq \Lam$. Let $T^{n-2}\subset B_{4r_a}(x_a)$ be an $(n-2)$-dimensional topological submanifold of $\R^{n+m}$, called the approximate singular set.
    Let $r_x:T^{n-2}\to \R_{\geq 0}$ be a nonnegative function satisfying $\Lip(r_x)\leq 2$, called the symmetry scale.
    We call $\Ann:=B_{r_a}(x_a)\setminus \ol B_{r_x}(T^{n-2})$ a $\delta$-annular region if the following hold:
    \begin{itemize}
        \item [(a1)] $T^{n-2}$ is $\delta$-graphical on $B_{r_x}(x)$ for all $x\in T^{n-2}$;
        \item [(a2)] For all $x\in T^{n-2}$ and $r>r_x$ such that $B_{r}(x)\subset B_{r_a}(x_a)$, there exists a $(n-2)$-dimensional affine subspace $L^{n-2}_{x,r}$ such that
        \begin{equation}
            d_{H}(L_{x,r}^{n-2}\cap B_r(x),T^{n-2}\cap B_r(x))<\delta r.
        \end{equation}
        \item [(a3)] For all $x\in T^{n-2}$ and $r>r_x$ such that $B_r(x)\subset B_{r_a}(x_a)$, we have that $M$ is $(n-2,\delta)$-symmetric on $B_r(x)$, but not $(n-1,\eps')$-symmetric on $B_r(x)$, where $\eps' = \eps'(n,m,\eps_{reg}) = \eps'(n,m,\Lam)$ is the constant in Theorem~\ref{A:thm:annular-decomp} when we take $\eps=\eps_{reg}$.
        \item [(a4)] For each $x\in T^{n-2}$, we have 
        \begin{equation}
            |\theta_{2r_a}(x)-\theta_{r_x}(x)|<\delta.
        \end{equation}
    \end{itemize}
\end{defn}
Here we say a submanifold $T^k\subset \R^{n+m}$ is $\delta$-graphical on $B_r(x)$ if there is an affine subspace $L_x$ with $d(x,L_x)<\delta r$ and a map $\mathfrak t_x:L_x\cap B_r(x)\to L_x^\perp$ such that $T^k\cap B_r(x) = \graph(\mathfrak t_x)$ and $|\mathfrak t_x| + r|\nabla \mathfrak t_x| + r^2 |\nabla^{(2)}\mathfrak t_x|\leq \delta r$.
We will often denote $\cH^{n-2}_T = \cH^{n-2}\mres T$.

It is of crucial importance in the proof of Theorem~\ref{thm:main-curvature} that $\cH^{n-2}_T$ satisfies the Ahlfors regularity below. For our purposes the rectifiability is of secondary importance, but we state it for completeness, referring again to the appendices for the expanded statements and proofs.

\begin{thm}[Annular Region Structure Theorem] \label{thm:annular-structure}
    Let $M$ be an area-minimizing current $\modt$ in $B_{8}\subset \R^{n+m}$ with $\ptl M \mres B_{8}=0$ and $|M|(B_8)\leq \Lam$, and $\Ann:=B_{4}\setminus \ol B_{r_x}(T^{n-2})$ a $\delta$-annular region.
    If $\delta\leq \delta(m,n,\Lam)$, then
    \begin{itemize}
        \item [(i)] (Volume) For all $x\in T^{n-2}$ and $r>0$ such that $B_{4r}(x)\subset B_4$ we have the Ahlfors regularity estimate
        \begin{equation}
            (1-C\sqrt{\delta})\ome_{n-2}r^{n-2}\leq \cH^{n-2}(B_r(x)\cap T^{n-2})\leq (1+C\sqrt{\delta})\ome_{n-2}r^{n-2},
        \end{equation}
        where $C=C(n,m,\Lam)>0$.

        \item [(ii)] (Rectifiability) The approximate singular set $T^{n-2}$ is rectifiable.
    \end{itemize}
\end{thm}

The final ingredient will allow us to cover the support of $M$ by $(n-1, \eps_{reg})$-symmetric regions (which satisfy Proposition~\ref{prop:symmetry-eps-regularity}) and annular regions (which will be treated by Theorem~\ref{thm:annular-curvature}). 
Moreover, the number of these regions are controlled (see (iii) below), so that we may sum the resulting estimates.

\begin{thm}[Annular Decomposition Theorem] \label{thm:annular-decomposition}
 Let $M$ be an area-minimizing current $\modt$ in $B_{64}\subset \R^{n+m}$ with $\ptl M  \mres B_{64}=0$ and $|M|(B_{64})\leq \Lam$. Then, for $\delta\leq \delta(n,m,\Lam)$ there exists $C=C(n,m,\Lam,\delta)>0$ such that we can write
 \begin{equation}
     \begin{aligned}
         B_1 \subset \sing(M)  \cup \bigcup_b B_{r_b}(x_b)\cup \bigcup_a \left(\Ann_a\cap B_{r_a}(x_a)\right)\\
         \sing(M)\cap B_1 \subset  \sT \cup \td \sS := \bigcup_a\left( T_a\cap B_{r_a}(x_a)\right)\cup \td \sS ,
     \end{aligned}
 \end{equation}
 where 
 \begin{itemize}
     \item [(i)] Each  $\Ann_a\subset B_{16r_a}(x_a)$ is a $\delta$-annular region with approximate singular set $T_a$;
     \item [(ii)] Each $B_{16r_b}(x_b)$ is $(n-1,\eps_{reg})$-symmetric;
     \item [(iii)] We have $\sum_{a} r_a^{n-2} + \sum_b r_b^{n-2}\leq C$;
     \item [(iv)] $\sT = \bigcup_a\left( T_a\cap B_{r_a}(x_a)\right)$ is a union of $(n-2)$-dimensional submanifolds with $\cH^{n-2}(\sT)\leq C$;
     \item [(v)] The residual singular set $\td \sS $ has vanishing measure $\cH^{n-2}(\td \sS)=0$.
 \end{itemize}

\end{thm}
\begin{proof} 
    This follows from Theorem~\ref{A:thm:annular-decomp} by taking $\eps=\eps_{reg}$. 
    On the other hand Proposition~\ref{prop:symmetry-eps-regularity} implies $\sS^{n-2}_{\eps_{reg}}=  \sing(M)$  which yields the statement.
\end{proof}

\section{Annular Region Curvature Estimate}

We now state the main estimate of this paper, which says that the $L^2$ norm of the second fundamental form is uniformly small on annular regions. 

\begin{thm}\label{thm:annular-curvature}
    Let $M$ be an $n$-dimensional area-minimizing current $\modt$ in $B_{32}$ with $\ptl M\mres B_{8}=0$ and $|M|(B_{32})\leq \Lambda$. Let $\Ann=B_{16}\setminus \ol B_{r_x}(T^{n-2})$ be a $\delta$-annular region and $\eps>0$. Then, for $\delta<\delta(n,m,\Lambda,\eps)$ we have 
    \begin{equation}
        \int_{B_{1}\cap \Ann} |A_M|^2 d|M|\leq C(n,m,\Lambda)(\eps + \delta).
    \end{equation}
\end{thm}

The proof of Theorem~\ref{thm:annular-curvature} will be split up into several parts. 
In Lemma~\ref{lem:pointwise-curvature} we use minimality to obtain a pointwise inequality for the full norm $|A_M|^2$ in terms of the other $n-1$ components. 
Then, in Section~\ref{subsec:jacobi-fields} we show how the components in these $(n-1)$ directions may be estimated by the norm of $(n-1)$ radial Jacobi fields. 
Finally, in Proposition~\ref{prop:local-annular-curvature} we show how to appropriately choose  $(n-1)$ Jacobi fields whose norm may be estimated by the volume density drop along the approximate singular set. Together with a Whitney decomposition of the annular region we obtain the proof of Theorem~\ref{thm:annular-curvature}.

\subsection{Pointwise Curvature Estimate}

\begin{lem} \label{lem:pointwise-curvature}
 Let $P^n$ be an $n$-dimensional linear subspace of $\R^{n+m}$ and $P^\perp$ its orthogonal complement. 
   Let $w_0,w_1,\ldots, w_{n-2}\in \R^{n+m}$ be in $\frac{1}{2}$-general position in $B_{r}\subset \R^{n+m}$
   and $L:= \Span\{w_{n-2}-w_0,\ldots, w_1-w_0\}$. 
   Let $L^\perp$ be the orthogonal complement for $L$ in $\R^{n+m}$. 
   Let $H= \pi_{P}(L)$ and $V = L^\perp \cap P = H^{\perp_P}$ be the orthogonal complement of $H$ in $P$. 
   Finally, let $A\in \mrm{Sym}^2(P^*)\otimes P^\perp$ be a symmetric bilinear form. 
   Suppose $\|\pi_{P^\perp}|_L\| <\frac{1}{2}$. Then, 
   \begin{itemize}
       \item [(i)] $H$ is $(n-2)$-dimensional, $V$ is $2$-dimensional, and $H\oplus V  = P$.  Moreover, $\{\pi_{P}(w_i-w_0)\}_{i=1}^{n-2}$ form a basis for $H$.
       \item [(ii)] For any orthonormal basis $\{h_i\}_{i=1}^{n-2}$ for $H$, we have 
       \begin{equation}
           \sum_{i=1}^{n-2} |A(\cdot, h_i)|^2  \leq \frac{C(n)}{r^{2}} \left( \sum_{i=1}^{n-2} |A(\cdot, \pi_P(w_i-w_0))|^2 \right)
       \end{equation}

       \item [(iii)] If moreover $\Tr_{P} A =0$, then for any orthonormal basis $\{h_i\}_{i=1}^{n-2}$ for $H$ and any unit vector $e\in V$,
       \begin{equation}
           |A|^2 \leq n \left(\sum_{i=1}^{n-2} |A(\cdot, h_i)|^2 + |A(\cdot, e)|^2\right).
       \end{equation}
       Consequently,
       \begin{equation}
           |A|^2 \leq C(n) \left(r^{-2} \sum_{i=1}^{n-2} |A(\cdot, \pi_P(w_i-w_0))|^2 + |A(\cdot, e)|^2 \right).
       \end{equation}
   \end{itemize}   
\end{lem}
\begin{proof}   

 We verify (i) and (ii) as  an elementary application of linear algebra. For $n=2$ there is nothing to prove so we assume $n\geq 3$.
 Set $v_i:=w_{i}-w_0$ for $i=1,\ldots, n-2$. Since  $\|\pi_{P^\perp}|_L\|<\frac{1}{2}$, for every $v\in L$ we have 
 \begin{equation}
     |\pi_P v|^2 = |v|^2 - |\pi_{P^\perp} v|^2 \geq (1-\frac{1}{4})|v|^2 .
 \end{equation}
 Thus the restriction $\pi_P|_{L}$ is injective and $\dim H = \dim \pi_P(L) = \dim L = n-2$ and $\pi_P v_1,\ldots, \pi_P v_{n-2}$ form a basis for $H$. 
 For $z\in P$, we know that $z$ is orthogonal to $H$ if and only if $\langle z,\pi_P v\rangle =0$ for all $v\in L$, which is equivalent to $\langle z,v\rangle =0$ for all $v\in L$. Thus $V= H^{\perp_P} = L^\perp \cap P$, which proves (i).

For (ii),  let $W$ be the matrix with columns equal to the vectors $v_1,\ldots, v_{n-2}$. 
Since $|v_i|<2r$ and $\{w_i\}$ are in $\frac{1}{2}$-general position, the smallest singular value $\sigma(W)$ satisfies $\sigma(W) >c(n) r$ for some constant $c(n)$.
 Thus for any vector $v\in \R^{n-2}$,
\begin{equation*}
    | \pi_P W v|\geq \frac{3}{4}|Wv|\geq c(n) r |v|
\end{equation*}
and we have the estimate $\sigma(\pi_PW)\geq c(n) r$ for the smallest singular value.
For any orthonormal basis $h_1,\ldots, h_{n-2}$ of $H$, we may view $A$ as a linear map $ \ol A:H\to \Hom(P,P^\perp)$ defined by $H\ni h \mapsto A(\cdot, h)$. Viewed in this way, $|\ol A|^2 = \sum_{i=1}^{n-2} |A(\cdot, h_i)|^2$. Since $| \ol A \pi_P W|^2 = \sum_{i=1}^{n-2} |A(\cdot, \pi_P v_i)|^2$, we estimate 
\begin{equation*}
    \sum_{i=1}^{n-2} |A(\cdot, \pi_P v_i)|^2\geq \sigma(\pi_P W)^2 |\ol A|^2 \geq c(n) r^2\sum_{i=1}^{n-2}|A(\cdot, h_i)|^2
\end{equation*}
as desired. 

We now prove (iii). Let $f\in V$ be a unit vector perpendicular to $e$, so that $h_1,\ldots, h_{n-2}, e, f$ form an orthonormal basis of $P$. For simplicity set $h_{n-1} := e$. By trace-freeness,
\begin{equation*}
    A(f,f) = - \sum_{i=1}^{n-1} A(h_i, h_i).
\end{equation*}
Together with symmetry we obtain
\begin{equation*}
    |A(\cdot ,f)|^2 \leq (n-1) \sum_{i=1}^{n-1}\left( |A(f, h_i)|^2 + |A(h_i,h_i)|^2\right) \leq (n-1)\sum_{i=1}^{n-1} |A(\cdot, h_i)|^2.
\end{equation*}
Thus,
\begin{equation*}
    |A|^2 \leq \sum_{i=1}^{n-1} |A(\cdot, h_i)|^2 + |A(\cdot, f)|^2 \leq n\sum_{i=1}^{n-1} |A(\cdot, h_i)|^2.
\end{equation*}
This finishes the proof.

\end{proof}

\subsection{Jacobi Fields} \label{subsec:jacobi-fields}
Let $M$ be a smooth minimal immersion with no boundary in an open set $U\subset \R^{n+m}$ and $\nabla_M$, $\nabla_M^\perp$ denote the Levi-Civita and normal connection on $M$.  Let $\mc J_M$ be the Jacobi operator on normal vector fields to $M$, so that $\mc J_M$ has the form
\begin{equation}
   \mc J_MZ := \Delta^\perp_M Z + \td A_M(Z)
\end{equation}
for all normal vector fields $Z$, 
where 
\begin{equation}
    \Delta_M^\perp Z = \sum_{i=1}^n \left(\nabla_{M,E_i}^\perp \nabla_{M,E_i}^\perp Z - \nabla_{M, (\nabla_{M,E_i}E_i)}^\perp Z\right)
\end{equation}
and
\begin{equation}
    \td A_M(Z) := \sum_{i,j=1}^n \langle A_{M}(E_i,E_j), Z\rangle A_M(E_i,E_j)
\end{equation}
whenever $\{E_i\}$ is a local orthonormal frame for $M$. Recall that $Z$ is called a Jacobi field if $\mc J_M Z =0$.

\begin{lem} \label{lem:jacobi-radial-trans}
      Let $M^n\subset \R^{n+m}$ be an $n$-dimensional minimal immersion with no boundary in an open set $U\subset \R^{n+m}$. 
     Let $Y$ be a vector field tangent to $M$.
     Fix $x_0\in \R^{n+m}$ and consider the vector field $X(x)=x-x_0$ and let $Z(x):= \pi_{T_xM}^\perp (x-x_0)$ be the projection onto the normal space of $M$. Then $\mc J_M Z=0$ and
      \begin{equation}
          \nabla^\perp_{M, Y} Z = -A_M(Y,\pi_{T_xM}(x-x_0)) 
      \end{equation}  
\end{lem}
\begin{proof}
    The flow generated by $X$ is a dilation centered at $x_0$, which preserves minimality. It follows that the normal variation field $Z$ is Jacobi. 
    Taking normal components in $\nabla_{\R^{n+m}, Y} X = Y$ gives
    \begin{equation*}
        0 = A_M(Y, \pi_{T_xM}(x-x_0)) + \nabla_{M,Y}^\perp Z
    \end{equation*}
    as desired.
\end{proof}

\begin{lem}\label{lem:jacobi-cacc}
    Let $Z$ be a normal Jacobi field on a smooth minimal immersion $M$ and $\phi$ a compactly supported Lipschitz function with $0\leq \phi\leq 1$. Suppose $M$ has no boundary in an open set containing $\spt \phi$.
    Then,
    \begin{equation}
        \int_M \phi^2 |\nabla_M^\perp Z|^2 \leq 4\int_M |\nabla_M \phi|^2 |Z|^2 +2 \int_M \phi^2 |A_M|^2 |Z|^2.
    \end{equation}
    In particular if $|A_M|\leq Ks^{-1}$ on $\spt \phi$ and $|\nabla \phi|\leq C s^{-1}$, then 
    \begin{equation}
        \int_M \phi^2 |\nabla_M^\perp Z|^2 \leq C(1+K^2)s^{-2}\int_{M\cap \spt \phi}|Z|^2.
    \end{equation}
\end{lem}
\begin{proof}
    Testing the Jacobi equation with $\phi^2 Z$ and integrating by parts, we obtain
    \begin{equation*}
        \int_M \phi^2 |\nabla_M^\perp Z|^2  = -2\int_M \phi  \langle \nabla_{M,\nabla_M \phi}^\perp Z,Z\rangle  + \int_M \phi^2 \langle \td A_M(Z),  Z\rangle.
    \end{equation*}
    Since $|\langle \td A_M(Z),  Z\rangle|\leq |A_M|^2 |Z|^2$, this yields
    \[
    \int_M \phi^2 |\nabla_M^\perp Z|^2 \leq 2 \int_M |\phi| |\nabla_M \phi| |\nabla_M^\perp Z| |Z| + \int_M \phi^2 |A_M|^2 |Z|^2.
    \]
    Applying Young's inequality and rearranging, we obtain the desired inequality.
\end{proof}

\subsection{Local Annular Curvature Estimate}

We will now estimate the $L^2$ norm of the second fundamental form by the density drop, integrated over the approximate singular set $T$. 
In the next section, we will use the monotonicity formula  to sum this pinching along an unbounded number of scales to prove Theorem \ref{thm:annular-curvature}.

\begin{prop} \label{prop:local-annular-curvature}
     Let $M$ be an $n$-dimensional area-minimizing current $\modt$ in $B_{32}$ with $\ptl M\mres B_{32}=0$ and $|M|(B_{32})\leq \Lambda$.
     Let $\Ann=B_{16}\setminus \ol B_{r_x}(T^{n-2})$ be a $\delta$-annular region and consider $y\in \Ann\cap B_1\cap \spt(M)$ with $r:=\dist(y,T^{n-2}) = |y-z|$ with $z\in T^{n-2}$, so that $r>r_z$.
     Then, for $\delta<\delta(n,m,\Lambda)$ we have 
     \begin{equation}
          \int_{B_{r/2}(y)}  |A_M|^2 d|M| \leq C(m,n,\Lambda) \int_{B_{r/4}(z)} |\theta_{2r}-\theta_r|(w) d \cH^{n-2}_T(w).
     \end{equation}
\end{prop}

The proof of Proposition~\ref{prop:local-annular-curvature} will follow from the pointwise curvature estimates in the previous section once we have constructed appropriate radial and translation  fields which are almost tangent to $M$ away from the approximate singular set. 
We will also require control of the density drop at these points and regularity away from the approximate singular set. 
Note that the curvature estimate in item~(iii) of the lemma below combined with the Annular Decomposition Theorem \ref{thm:annular-decomposition} implies that $A_M\in L^{2,\infty}$. This estimate has been known since the work of \cite{Naber2020} (see also \cite{Liu2020WeakCounterexamples}). 
The key point is that Proposition \ref{prop:local-annular-curvature} allows us to sum over scales and obtain the endpoint estimate $A_M\in L^2$.
The following lemma summarizes the geometric information necessary to prove Proposition~\ref{prop:local-annular-curvature}.

\begin{lem} \label{lem:local-annular-geometry}
    Let $M,\Ann,T^{n-2}, y,z,r$ be as in Proposition \ref{prop:local-annular-curvature}.
    There exists  $C=C(m,n,\Lambda)>0$ such that for every $\eta\in (0,1)$, if $\delta\leq \delta(n,m,\Lambda,\eta)$ then there exist points $w_0,\ldots, w_{n-2}\in T^{n-2}\cap B_{r/4}(z)$ such that
    \begin{itemize}
        \item [(i)]$\{w_0,\ldots, w_{n-2}\}$ are in $1/2$-general position in $B_{r/4}(z)$;

        \item [(ii)] We have 
        \begin{equation}
            |\theta_{2r}-\theta_r|(w_j)\leq C r^{2-n}\int_{B_{r/4}(z)} |\theta_{2r}-\theta_r|(w) d\cH_T^{n-2}(w) \quad \forall 0\leq j\leq n-2.
        \end{equation}

                \item [(iii)] Set $L:= \Span\{w_{n-2}-w_0,\ldots, w_1-w_0\}$. Then $M$ is smooth in $B_{3r/4}(y)$. Moreover, for every $x\in \spt(M)\cap B_{3r/4}(y)$, define 
                \begin{equation}
                    H_x := \pi_{T_xM}(L), \quad  V_x:=H_x^{\perp_{T_xM}}=T_xM\cap L^\perp.
                \end{equation} 
                Then, $H_x$ is $(n-2)$-dimensional, $V_x$ is $2$-dimensional,  and for all $x\in \spt(M)\cap B_{3r/4}(y)$ we have the estimates
        \begin{equation}
            | \pi_{T_xM}^\perp |_L|<\eta, \quad  |\pi_{V_x}(x-w_0)|\geq \frac{1}{8} r, \quad |A_M|(x)\leq \eta r^{-1}.
        \end{equation}

    \end{itemize}
\end{lem}

\begin{proof}
    Let $L_{z,r}^{n-2}$ be the affine $(n-2)$-plane as in Definition \ref{def:annular-region}(a2) and let $V_{z,r}^{n-2}$ be its vector subspace. 
    By Definition \ref{def:annular-region}(a2), there is $\ol z \in L_{z,r}$ with $|z-\ol z|<\delta r$. 
    Let $e_1,\ldots, e_{n-2}$ be an orthonormal basis of $V_{z,r}$, set $\ol p_0=\ol z$ and set
    \[
    \ol p_i := \ol z + \frac{3}{20} r e_i\in L_{z,r} \cap B_r(z), \quad i=1,\ldots, n-2
    \]
    for $\delta$ small enough. 
    By Definition \ref{def:annular-region}(a2) again, we may then choose points $p_1,\ldots, p_{n-2}\in T\cap B_r(z)$ such that $|p_i-\ol p_i|<\delta r$ for all $i=1,\ldots, n-2$. Set $p_0:=z$.

    Let $\beta,D>0$ be geometric constants which we will choose later and consider the set 
    \begin{equation*}
        G:= \left\{x\in B_{r/4}(z)\cap T^{n-2} \ssep |\theta_{2r}-\theta_r|(x)\leq D r^{2-n} \int_{B_{r/4}(z)}|\theta_{2r}-\theta_r|(w) d\cH_T^{n-2}(w)\right\}.
    \end{equation*}
    By Markov's inequality,
    \begin{equation*}
        \cH^{n-2}( (T^{n-2}\cap B_{r/4}(z))\setminus G)\leq \frac{r^{n-2}}{D}.
    \end{equation*}
    On the other hand, for $\delta\leq \beta$ and $\beta$ small enough ($\beta\leq 10^{-3}$ suffices), for every $i=0,\ldots, n-2$ we have that $B_{\beta r}(p_i)\subset B_{r/4}(z)$ and
    \begin{equation}
        \cH^{n-2}(T^{n-2}\cap B_{\beta r}(p_i))\geq (1-C(n,m,\Lambda)\sqrt{\delta})\ome_{n-2}(\beta r)^{n-2} \geq \frac{1}{2} \ome_{n-2} (\beta r)^{n-2}
    \end{equation}
    by the Ahlfors regularity of the Annular Structure Theorem \ref{thm:annular-structure}. 
    Choosing $D\geq \frac{2}{\ome_{n-2}\beta^{n-2}} +1$, we see that $B_{\beta r}(p_i)\cap G\neq \emptyset$ and so we may choose $w_i\in B_{\beta r}(p_i)\cap G$ for $i=0,\ldots, n-2$. 
    By construction, the $w_i$ satisfy (ii) and it is straightforward to see that one can choose $\beta=\beta(n)$ small enough so that $w_0,\ldots, w_{n-2}$ are in $\frac{1}{2}$-general position.


    The proof of (iii) will follow by a contradiction and compactness argument. If (iii) is not true, we can find $\eta\in (0,1)$ and sequences $M_{b},\cA_{b},T_b, y_b,z_b,r_b$ satisfying the assumptions in Proposition \ref{prop:local-annular-curvature} with $\delta_b\downarrow 0$. 
    For each $b$ we let $w_{b,0},\ldots, w_{b,n-2}\in T_{b}\cap B_{r_b/4}(z_b)$ be the corresponding points satisfying (i) and (ii). 

    Rescaling by $r_b$, translating by $z_b$, and relabeling, we may henceforth assume  
    \begin{equation}
        r_b=1, \quad z_b =0, \quad |y_b| = \dist(y_b, T_b)=1, \quad w_{b,i} \in B_{1/4}.
    \end{equation}
    Note that by the monotonicity formula the rescaled currents still satisfy $|M_b|(B_4)\leq \Lambda$. 
    Let $L_{b}$ be the affine $(n-2)$-plane given by Definition~\ref{def:annular-region}(a2) for $z_b=0, r_b=1$ and $V_b$ its vector subspace. Since $0=z_b\in T_b$, by Definition~\ref{def:annular-region}(a2) we have, after possibly taking a subsequence, that
       $ L_b\to L$ and  $ V_b\to L$
    in Hausdorff distance on $B_1$ for some $(n-2)$-subspace $L$.
    Passing to a further subsequence if necessary, we have 
        $
       y_b\to y$ and  $ w_{b,i}\to w_i \in L
    $ for $i=0,\ldots, n-2$.
    It follows that  $w_0,\ldots, w_{n-2}$ are in $1/4$-general position.  
    In particular, $L = w_0 + \Span\{w_{n-2}-w_0,\ldots, w_{1}-w_0\}$.
    Finally, by the uniform mass bounds we have (after another subsequence)
    $
        M_b \to M
    $ in the flat norm as area-minimizing currents $\modt$ and as stationary integral varifolds. 

    We claim that $M$ is conical with spine containing $L$ in $B_1$. Indeed, if $r_x^b$ denotes the symmetry scale for $M_b$, then $r_{w_{b,i}}^b \leq r_0^b + 2|w_{b,i}| < \frac{3}{2}$ since $\Lip(r_x^b)\leq 2$ and $r_0^b<1$.
    Letting $\theta^{M_b}_r(x)$ be the regularized density of $M_b$, we then have 
    \begin{equation*}
        |\theta^{M_b}_{5/2}(w_{b,i} ) - \theta^{M_b}_{3/2}(w_{b,i})|<\delta_b
    \end{equation*}
    for all $b$ and $i=0,\ldots, n-2$.
    By varifold convergence, we may pass to the limit and so obtain $\theta^M_{5/2}(w_i) = \theta^M_{3/2}(w_i)$. It follows that $M$ is conical and translation invariant along $L$ in $B_2$.

    By upper-semicontinuity, $y\in \spt(M)$ and $\dist(y,L)=1$. Since $M$ is an area-minimizing cone $\modt$ with spine dimension at least $n-2$, the classification of such area-minimizing cones $\modt$ \cite{Morgan1982} implies that in $\ol B_{4/5}(y)$, $M$ is a multiplicity-1 copy of a single $n$-plane. 
    By Allard's Theorem \cite{Allard1972}, in the ball $B_{4/5}(y)$,  $M_b$ is smooth, $M_b\to M$ smoothly, and
    \begin{equation}
        \sup_{x\in \spt (M_b)\cap B_{4/5}(y)} |A_{M_b}|(x)\to 0.
    \end{equation}
    Moreover, since $w_{b,i}\to w_i$ it follows that $\ol L_b:= \Span\{w_{b,i}-w_{b,0}\}_{i=1}^{n-2} \to L\subset T_xM$ and so  
    \begin{equation}
       \sup_{x\in \spt (M_b)\cap B_{4/5}(y)} \|\pi_{T_x M_b}^\perp|_{\ol L_b}\|\to 0
    \end{equation}
    We thus obtain the first and third estimates of (iii).  Moreover, for large enough $b$ we conclude $\pi_{T_xM_b|_{\ol L_b}}$ is injective so that $H_x^b$ has dimension $n-2$.

    For the remaining  estimate, let $x\in \spt(M)\cap B_{4/5}(y)$ and let $V_x:=T_xM \cap L^\perp$. We then have 
    \begin{equation}
    \begin{aligned}
        |\pi_{V_x}(x-w_0)|&=|\pi_{L^\perp}(x-w_0)| = \dist(x,L)\\
        &\geq \dist(y,L)-|x-y|\geq \frac{1}{5}
    \end{aligned}
    \end{equation}
    since $x-w_0\in T_xM$. By the smooth convergence, $\ol L_b\to L$, and $w_{b,0}\to w_0$, we obtain
    \begin{equation}
        |\pi_{T_{x}M_b\cap \ol L_b^\perp}(x-w_{b,0})|\geq \frac{1}{8}
    \end{equation}
    on $B_{3/4}(y_b)\cap \spt(M_b)$ for sufficiently large $b$.
\end{proof}

We may now finish the proof of Proposition~\ref{prop:local-annular-curvature}.

\begin{proof}[Proof of Proposition \ref{prop:local-annular-curvature}]
    Let $w_0,\ldots, w_{n-2}$ be the points in $B_{r/4}(z)$  in $1/2$-general position furnished by Lemma~\ref{lem:local-annular-geometry} with $\eta=\frac{1}{8}$. Let $v_i:= w_i-w_0$ for $i=1,\ldots, n-2$ and set $L= \Span\{v_1,\ldots, v_{n-2}\}$.
    Recalling that $M$ is smooth in $B_{3r/4}(y)$, for each $x\in \spt(M)\cap B_{3r/4}(y)$ we let $V_x = L^\perp\cap T_xM = \pi_{T_xM}(L)^{\perp_{T_xM}}$ be the orthogonal complement of $H_x:= \pi_{T_xM}(L)$ in $T_xM$.  
    For $i=0,\ldots, n-2$, set $Z_i(x) = \pi_{T_xM}^\perp (x-w_i)$ and write for each $x\in B_{3r/4}(y)\cap \spt(M)$,
    \begin{equation}
        \pi_{T_xM} (x-w_0) = h(x) + v(x), \quad h(x)\in H_x, \,v(x)\in V_x
    \end{equation}
    and let $e(x) = \frac{v(x)}{|v(x)|}$.
    By Lemma~\ref{lem:pointwise-curvature} we have for $x\in \spt(M)\cap B_{3r/4}(y)$,
    \begin{equation} \label{eq:sff-two-terms}
        |A_M|^2(x) \leq C \left( r^{-2}\sum_{i=1}^{n-2}|A_M(\cdot, \pi_{T_xM}v_i)|^2 + |A_M(\cdot,e(x))|^2\right).
    \end{equation}
    Let us begin by controlling the second  term in \ref{eq:sff-two-terms}. 
    By Lemma~\ref{lem:jacobi-radial-trans},
    \begin{equation}
        \nabla^\perp_M Z_0(x) = -A_M(\cdot, \pi_{T_{x}M}(x-w_0)) = - A_M(\cdot, h(x)) - |v(x)| A_M(\cdot, e(x)).
    \end{equation}
    so that 
    \begin{equation}
        |v(x)|^2 |A_M(\cdot, e(x))|^2 \leq C\left(  |\nabla_M^\perp Z_0|^2 + |A_M(\cdot,h(x))|^2\right)
    \end{equation}
    By Lemma~\ref{lem:local-annular-geometry}, $|v(x)|\geq \frac{1}{8}r$ and thus
    \begin{equation} \label{eq:sff-e-bound}
    \begin{aligned}
        |A_M(\cdot, e(x))|^2 &\leq Cr^{-2} \left( |\nabla_M^\perp Z_0|^2 + |A_M(\cdot,h(x))|^2 \right) \\
        &\leq 
        Cr^{-2} \left(|\nabla_M^\perp Z_0|^2+  \sum_{i=1}^{n-2}|A_M(\cdot, \pi_{T_xM}v_i)|^2 \right)
    \end{aligned}
    \end{equation}
    where in the second line we have used that $|h(x)|\leq Cr$, $\|\pi_{T_xM}^\perp|_L\|<\frac{1}{8}$, and $w_0,\ldots, w_{n-2}$ are in $1/2$-general position at scale $r/4$. 

    In view of equations \ref{eq:sff-two-terms} and \ref{eq:sff-e-bound}, it suffices to estimate $|\nabla_M^\perp Z_0|^2$ and $|A_M(\cdot, \pi_{T_xM}v_i)|^2$  for $i=1,\ldots, n-2$. On the other hand writing $v_i = (x-w_0) - (x-w_i)$ we have 
    \begin{equation} \label{eq:sff-transverse-nabla-perp}
    \begin{aligned}
        |A_M(\cdot, \pi_{T_xM}v_i)|^2 &\leq C\left(|A_M(\cdot, \pi_{T_xM}(x-w_i))|^2+|A_M(\cdot, \pi_{T_xM}(x-w_0))|^2 \right) \\
        &\leq C \left( |\nabla^\perp_M Z_i|^2 + |\nabla^\perp_M Z_0|^2 \right)
    \end{aligned}
    \end{equation}
    by Lemma~\ref{lem:jacobi-radial-trans}. Thus we only need to estimate $\int_{B_{r/2}(y)}|\nabla_M^\perp Z_i|^2$ for $i=0,\ldots, n-2$.

    Let $\phi:B_{5r/8}(y)\to [0,1]$ be a smooth compactly supported bump function with $\phi \equiv 1$ 
    on $B_{r/2}(y)$ and such that $|\nabla \phi|\leq C r^{-1}$.
    For $i=0,\ldots, n-2$, we may therefore apply Lemma \ref{lem:jacobi-cacc} to obtain,
    \begin{equation} \label{eq:wi-pinching-controls}
    \begin{aligned}
        r^{-n} \int_{B_{r/2}(y)} |\nabla_M^\perp Z_i|^2 d|M|(x) &\leq r^{-n} \int_{B_{5r/8}(y)} \phi^2 |\nabla_M^\perp Z_i|^2 d|M|(x) \\
        &\leq C r^{-n-2}\int_{B_{5r/8}(y)} |\pi_{T_xM}^\perp(x-w_i)|^2 d|M|(x) \\
        &\leq C|\theta_{2r}(w_i)-\theta_r(w_i)|
    \end{aligned}
    \end{equation}
    where we have used Proposition~\ref{prop:monotonicity} in the last line.
    Putting together equations \ref{eq:sff-two-terms}, \ref{eq:sff-e-bound}, \ref{eq:sff-transverse-nabla-perp}, \ref{eq:wi-pinching-controls},  integrating over $B_{r/2}(y)$, and recalling Lemma~\ref{lem:local-annular-geometry}(ii), we obtain
    \begin{equation}
    \begin{aligned}
        r^{2-n} \int_{B_{r/2}(y)} |A_M|^2 d|M|(x) & \leq C r^{-n} \sum_{i=0}^{n-2} \int_{B_{r/2}(y)} |\nabla_M^\perp Z_i|^2 d|M|(x) \\
        &\leq C \sum_{i=0}^{n-2} |\theta_{2r}(w_i) - \theta_r(w_i)|\\
        &\leq C r^{2-n} \int_{B_{r/4}(z)} |\theta_{2r}-\theta_r|(w) d\cH_T^{n-2}(w).
    \end{aligned}
    \end{equation}

\end{proof}

\subsection{Proof of Theorem~\ref{thm:annular-curvature}}

Given Proposition~\ref{prop:local-annular-curvature}, the proof of Theorem~\ref{thm:annular-curvature}  follows from a Whitney decomposition of the annular region. For an essentially equivalent cover on annular regions see \cite[\S4.4]{JiangNaber2021}.

\begin{proof}[Proof of Theorem~\ref{thm:annular-curvature}]

We take a Whitney decomposition $\mc W$ of $\R^{n+m}\setminus T^{n-2}$ into dyadic closed cubes $Q\in \mc W$ satisfying
\begin{equation} \label{eq:whitney-diam}
    3\diam Q \leq \dist(Q,T^{n-2})\leq 9\diam(Q).
\end{equation}
Indeed, we may take $\mc W$ to be the maximal  (under inclusion) subcollection of dyadic cubes $Q$ in $\R^{n+m}$ which satisfy
\begin{equation}
    3\diam (Q) \leq \dist(Q,T^{n-2}).
\end{equation}
It is now straightforward to verify that $\mc W$ is a Whitney decomposition satisfying \ref{eq:whitney-diam}.

Set
\begin{equation}
    \mc W_{\cA}:= \{Q\in \mc W \ssep Q\cap B_1\cap\cA\cap \spt(M)\neq \emptyset\}
\end{equation}
and for every $Q\in \mc W_{\cA}$ choose $y_Q\in Q\cap B_1\cap \cA\cap \spt(M)$ and set 
\begin{equation}
    d_Q= \dist(y_Q,T^{n-2})  = |y_Q-z_Q|
\end{equation}
with $z_Q\in T^{n-2}$.
By the lower bound in \ref{eq:whitney-diam}, $Q\subset B_{\diam(Q)}(y_Q)\subset B_{d_Q/2}(y_Q)$ and thus the collection $\{B_{d_Q/2}(y_Q)\}_{Q\in \mc W_\cA}$ covers $B_1\cap \cA \cap \spt(M)$. 

Set, for each $Q\in \mc W_{\cA}$, $R_Q :=T^{n-2}\cap B_{d_Q/4}(z_Q)$. We claim that for each $w\in T^{n-2}$ and $t>0$, we have 
\begin{equation} \label{eq:whitney-packing}
     \# \mc W_{w,t} := \#\left\{Q\in \mc W_\cA \ssep w\in R_Q , \,t\in (d_Q, 2d_Q] \right\} \leq C(n,m).
\end{equation}
Indeed, if $Q\in \mc W_{w,t}$ then 
\begin{equation}
    \dist(w,Q)\leq |w-y_Q|\leq |w-z_Q| + |z_Q-y_Q|\leq 2d_Q \leq 2t.
\end{equation}
Since $\diam(Q) \leq C(n,m)d_Q \leq C(n,m)t$, it follows that every $Q\in \mc W_{w,t}$ is contained in $B_{C(m,n)t}(w)$. 
 Since the $Q$'s have pairwise disjoint interiors with diameter comparable to $t$, a packing estimate proves the claim. 

We have 
\begin{equation} \label{eq:sff-neck-estimate-split}
    \begin{aligned}
        \int_{\cA\cap B_1} |A_M|^2 d|M| &\leq \sum_{Q\in \mc W_{\mc A}} \int_{B_{d_Q/2}(y_Q)} |A_M|^2 d|M| \\
        &\leq \sum_{\substack{Q\in \mc W_{\mc A}\\ d_Q>4r_{z_Q}} }\int_{B_{d_Q/2}(y_Q)} |A_M|^2 d|M| + \sum_{\substack{Q\in \mc W_{\mc A}\\ d_Q>r_{z_Q}\geq d_Q/4}} \int_{B_{d_Q/2}(y_Q)} |A_M|^2 d|M|. 
    \end{aligned}
\end{equation}
For the first sum in \ref{eq:sff-neck-estimate-split}, we note that for every $Q\in \mc W_{\cA}$ with $d_Q>4r_{z_Q}$ and $w\in R_Q= T^{n-2}\cap B_{d_Q/4}(z_Q)$, we have 
\(
    r_w\leq r_{z_Q} + \frac{d_Q}{2}<d_Q 
\)
since $\Lip(r_z)\leq 2$. 
On the other hand for fixed $w\in T^{n-2}$ consider the monotone function $\theta_{r}(w)$ and let $\nu_w$ be its associated  Lebesgue–Stieltjes measure obeying $\nu_w((s,r]) = \theta_r(w)-\theta_s(w)$. By \ref{eq:whitney-packing},
\begin{equation} \label{eq:whitney-packing-indicator}
    m_w(t):= \sum_{\substack{Q\in \mc W_\cA \\ w\in R_Q}} 1_{(d_Q, 2d_Q]}(t) \leq C(n,m) 
\end{equation}
for all $w\in T^{n-2}\cap B_1$ and $t>0$.
Thus we apply Lemma~\ref{lem:pointwise-curvature}, Proposition~\ref{prop:local-annular-curvature} and \ref{eq:whitney-packing-indicator} to estimate
\begin{equation} \label{eq:neck-est-first-term}
    \begin{aligned}
        \sum_{\substack{Q\in \mc W_{\mc A}\\ d_Q>4r_{z_Q}} }\int_{B_{d_Q/2}(y_Q)} |A_M|^2 d|M|& \leq C  \sum_{\substack{Q\in \mc W_{\mc A}\\ d_Q>4r_{z_Q}} } \int_{R_Q} |\theta_{2d_Q}-\theta_{d_Q}|(w)d\cH^{n-2}_T(w)\\
        &\leq C \int_{T^{n-2}\cap B_3} \int_{r_w}^{4} m_w(t) d\nu_w d\cH_T^{n-2}(w) 
        \leq C \delta
    \end{aligned}
\end{equation}
where we have applied Definition~\ref{def:annular-region}(a4) and the Ahlfors regularity Theorem~\ref{thm:annular-structure} in the last line.

To estimate the second sum in \ref{eq:sff-neck-estimate-split}, we use Lemma~\ref{lem:local-annular-geometry} to estimate, with $\delta<\delta(n,m,\Lam,\eps)$ small enough so that $\eta^2 <\eps$,
\begin{equation} \label{eq:inner-neck-est}
    \sum_{\substack{Q\in \mc W_{\mc A}\\ d_Q>r_{z_Q}\geq d_Q/4}} \int_{B_{d_Q/2}(y_Q)} |A_M|^2 d|M| \leq C \eps  \sum_{\substack{Q\in \mc W_{\mc A}\\ d_Q>r_{z_Q}\geq d_Q/4}} d_Q^{n-2}.
\end{equation}
Define for each $w\in T^{n-2}$,
\begin{equation}
    N(w) := \sum_{\substack{Q\in \mc W_{\mc A}\\ d_Q>r_{z_Q}\geq d_Q/4}} 1_{T^{n-2}\cap B_{r_{z_Q}/8}(z_Q)}(w).
\end{equation}
For every $Q\in \mc W_\cA$ with $d_Q>r_{z_Q}\geq d_Q/4$, any $w\in T^{n-2}\cap B_{r_{z_Q}/8}(z_Q)$ satisfies 
$
    \frac{4}{5}r_{w}<d_Q \leq \frac{16}{3}r_w.
$
Therefore for any $w\in T^{n-2}$,
\begin{equation}
\begin{aligned}
    N(w) & \leq C  \int_0^\infty \sum_{\substack{ Q\in \mc W_{\mc A},\, w\in T^{n-2}\cap B_{r_{z_Q}/8}(z_Q) \\ d_Q>r_{z_Q}\geq d_Q/4}} 1_{(d_Q, 2d_Q]}(t) \frac{dt}{t} 
    \leq \int_{\frac{4}{5}r_w}^{\frac{32}{3}r_w} m_w(t) \frac{dt}{t} \leq C(n,m,\Lam)
\end{aligned}
\end{equation}
since $T^{n-2}\cap B_{r_{z_Q}}(z_Q)\subset R_Q$ and \ref{eq:whitney-packing-indicator} holds.

Finally, by the Ahlfors regularity in Theorem~\ref{thm:annular-structure} 
\begin{equation}
    \begin{aligned}
        \sum_{\substack{Q\in \mc W_{\mc A}\\ d_Q>r_{z_Q}\geq d_Q/4}} d_Q^{n-2} &\leq C \sum_{\substack{Q\in \mc W_{\mc A}\\ d_Q>r_{z_Q}\geq d_Q/4}}  \cH^{n-2}_T(B_{r_{z_Q}/8}(z_Q) )
        \leq C \int N(w) d\cH^{n-2}_T(w) \leq C.
    \end{aligned}
\end{equation}
Recalling \ref{eq:neck-est-first-term} and \ref{eq:inner-neck-est}, we finish the proof of the theorem.  
\end{proof}

\section{Proof of Main Theorems}

By applying Theorem~\ref{thm:annular-curvature} and Theorem~\ref{thm:annular-decomposition}, we may finish the proof of Theorem~\ref{thm:main-curvature}.
\begin{proof}[Proof of Theorem~\ref{thm:main-curvature}]
    By Theorem~\ref{thm:annular-decomposition} we have 
    \begin{equation}
        \reg(M) \subset \bigcup_b B_{r_b}(x_b) \cup \bigcup_a \left(\Ann_a\cap B_{r_a}(x_a)\right)
    \end{equation}
    where $B_{16r_b}(x_b)$ are $(n-1,\eps_{reg})$-symmetric and $\Ann_a\cap B_{16r_a}(x_a)$ are $\delta$-annular regions, for $\delta$ small enough.
    Applying Proposition~\ref{prop:symmetry-eps-regularity} we obtain
    \begin{equation}
        \int_{B_{r_b}(x_b)} |A_{M}|^2 d|M| \leq C r_b^{n-2}.
    \end{equation}
    Meanwhile the (scale invariant form of) Theorem~\ref{thm:annular-curvature} gives
    \begin{equation}
        \int_{\Ann_a \cap B_{r_a}(x_a)}|A_M|^2 d|M| \leq r_a^{n-2} C(\eps+\delta)\leq Cr_a^{n-2}.
    \end{equation}
    where $C=C(n,m,\Lam)$ in both cases.
   Combining these estimates with Theorem~\ref{thm:annular-decomposition}(iii) gives
   \begin{equation}
   \begin{aligned}
       \int_{B_1} |A_M|^2 d|M| &\leq \sum_a \int_{\Ann_a \cap B_{r_a}(x_a)}|A_M|^2 d|M| + \sum_b  \int_{B_{r_b}(x_b)} |A_{M}|^2 d|M|  \\
       &\leq C\left( \sum_a r_a^{n-2} + \sum_b r_b^{n-2} \right)\leq C
    \end{aligned}
   \end{equation}
   which finishes the proof.
\end{proof}

\begin{proof}[Proof of Corollary~\ref{cor:2d-topology-bound}]

We first observe that Theorem~\ref{thm:annular-decomposition} shows that $\# (\sing(M)\cap B_3)\leq N(m,\Lam)$. If $\kappa_\Gamma$ denotes the  curvature of a curve $\Gamma$, then by Sard's theorem if we consider the smooth curve $\Gamma_t=\reg(M)\cap  \ptl B_t$ for a.e. $t\in (1,2)$, we have 
\begin{equation*}
    |\kappa_{\Gamma_t}| |\nabla_M |x||\leq t^{-1}  + 2|A_M|.
\end{equation*}
By coarea, Cauchy-Schwarz, and \ref{eq:main-L2-bound} we then have 
\begin{equation*}
    \int_1^2 \int_{\Gamma_t} |\kappa_{\Gamma_t}| d\cH^1 dt\leq \int_{\reg(M)\cap (B_2\setminus B_1)} (|x|^{-1} + 2|A_M|) d|M|\leq \Lam + 2\sqrt{\Lam C}\leq C.
\end{equation*}
We  choose $r\in (1,2)$ such that $\ptl B_r$ avoids the singular set, meets $\reg(M)$ transversely, and satisfies $\int_{\Gamma_r} |\kappa_r|\leq C$. 

On the other hand by choosing $\eta = 1/10$, $\delta$ as in Lemma~\ref{lem:local-annular-geometry}, and applying Theorem~\ref{thm:annular-decomposition}, we see that each singular point $p\in \sing(M)\cap B_r$ is a point in the quantitative singular set of some annular region and, near $p$, $|A_M|(x)\leq \frac{1}{10}\dist(x,p)^{-1}$ with $\spt(M) \cap B_{\eps}(p)$ diffeomorphic to at most $(m+2)/2$ smooth local embedded submanifolds  intersecting only at $p$ by \cite[Corollary 7]{Morgan1982} for $\eps$ small enough. 
In particular we have $\int_{\reg(M)\cap \ptl B_\eps(p)}|\kappa_{\reg(M)\cap \ptl B_\eps(p)}|  \leq 2\pi (m+2)$. 

Setting $F = (\spt(M)\cap \ol B_r)\setminus \bigcup_{p\in \sing(M)\cap B_r} B_\eps(p)$ we see that the total curvature of $\ptl F$ is bounded by a uniform constant $C(m,\Lam)$. 
Since no compact boundaryless minimal surfaces can exist in Euclidean space, we obtain by Fenchel's theorem that $b_0(F)\leq b_0(\ptl F)\leq C\int_{\ptl F} |\kappa_{\ptl F}|\leq C$. 
By Gauss-Bonnet,
\begin{equation*}
    2\pi \chi(F) =-\frac{1}{2}\int_F |A_M|^2 d|M| +\int_{\ptl F}k_{\ptl F} \geq -C
\end{equation*}
where $k_{\ptl F}$ is the geodesic curvature. Therefore $b_1(F;\Z_2)=b_0(F)-\chi(F)\leq C(m,\Lam)$. It follows from the classification of surfaces that there are at most $N(m,\Lam)$  diffeomorphism types for $F$. Since $\reg(M)\cap B_r$ is diffeomorphic to the interior of $F$, the corollary is proved.

\end{proof}

\begin{rmk}
    Corollary~\ref{cor:2d-topology-bound} may also be proved by gluing together annular and $(1,\eps_{reg})$-symmetric regions as in \cite[\S8]{Cheeger2015}, using Gauss-Bonnet to show that annular regions are  diffeomorphic to annuli.
\end{rmk}

\appendix

\section{Annular Regions} \label{sec:A:annular}

In this appendix  we  outline the annular decomposition and structure theory for an integer rectifiable stationary varifold $M$. 
The presentation follows closely that of \cite{ChowJiangNaber}. 
In what follows we let $\theta_r(x)$ denote the regularized mass density for $M$ as in Definition~\ref{def:regularized-mass}.

\subsection{Annular Structure Theorem}

We first give the more general definition of annular regions.
\begin{defn}[Annular Region]
        Let $M$ be a stationary integer rectifiable varifold in $B_{4r_a}(x_a)$. Let $T^{k}\subset B_{4r_a}(x_a)$ be a $k$-dimensional topological submanifold of $\R^{n+m}$, called the approximate singular set.
    Let $r_x:T^{k}\to \R_{\geq 0}$ be a nonnegative function satisfying $\Lip(r_x)\leq 2$, called the symmetry scale.
    We call $\Ann:=B_{r_a}(x_a)\setminus \ol B_{r_x}(T^{k})$ a $(k,\eps,\delta)$-annular region if the following hold:
    \begin{itemize}
        \item [(a1)] $T^{k}$ is $\delta$-graphical on $B_{r_x}(x)$ for all $x\in T^{k}$;
        \item [(a2)] For all $x\in T^{k}$ and $r>r_x$ such that $B_{r}(x)\subset B_{r_a}(x_a)$, there exists a $k$-dimensional affine subspace $L^{k}_{x,r}$ such that
        \begin{equation}
            d_{H}(L_{x,r}^{k}\cap B_r(x),T^{k}\cap B_r(x))<\delta r.
        \end{equation}
        \item [(a3)] For all $x\in T^{k}$ and $r>r_x$ such that $B_r(x)\subset B_{r_a}(x_a)$, we have that $M$ is $(k,\delta)$-symmetric on $B_r(x)$, but not $(k+1,\eps)$-symmetric on $B_r(x)$.
        \item [(a4)] For each $x\in T^{k}$, we have 
        \begin{equation}
            |\theta_{2r_a}(x)-\theta_{r_x}(x)|<\delta.
        \end{equation}
    \end{itemize}
\end{defn}

\begin{thm}[Annular Structure Theorem] \label{A:thm:annular-structure}
    Let $M$ be a stationary integral varifold in $B_8$ with $|M|(B_8)\leq \Lam$ and let $\cA = B_4\setminus \ol B_{r_x}(T^k)$ be a $(k,\eps,\delta)$-annular region. If $\delta\leq \delta(n,m,\Lam,\eps)$, then 
    \begin{enumerate}
        \item [(1)] For all $x\in T^k$ and $r>0$ such that $B_{4r}(x)\subset B_4$, we have 
        \begin{equation}
            (1-C\sqrt{\delta})\ome_k r^k\leq \cH^k(B_r(x)\cap T^k)\leq (1+C\sqrt{\delta})\ome_k r^k,
        \end{equation}
        where $C=C(n,m,\Lam,\eps)>0$.

        \item [(2)] The approximate singular set $T^k$ is $k$-rectifiable.
    \end{enumerate}
\end{thm}

The only input  for the proof of the annular structure theorem which is specific to stationary varifolds  is the $L^2$ subspace approximation theorem of \cite{Naber2020}. 
We recall that for any measure  $\mu$ on $\R^{n+m}$ and $k\in \N$,   the $L^2$ Jones $\beta$-number is defined by 
\begin{equation}
    \beta_k(x,r;\mu)^2 = \inf_{L^k \text{ affine $k$-plane} } r^{-k-2}\int_{B_r(x)} \dist(y,L^k)^2 d\mu(y)
\end{equation}

We can now state the $L^2$ Subspace Approximation Theorem, the proof of which is a straightforward modification of the arguments in \cite{ChowJiangNaber} (see also \cite{Naber2020}).
\begin{thm}[$L^2$ Subspace Approximation] \label{A:thm:L2-subspace-approx}
     Let $M$ be a stationary integral varifold in $B_2$ with $|M|(B_2)\leq \Lam$. Assume that $M$ is $(k,\delta)$-symmetric  on $B_{10r}(x)\subset B_2$ but not $(k+1,\eps)$-symmetric on $B_{2r}(x)$. Then for $\delta\leq \delta(n,m,\eps)$, there exists $C=C(n,m,\Lam,\eps)>0$ such that for any finite measure $\mu$ on $B_1$, we have 
     \begin{equation}
         \beta_k(x,r;\mu)^2 \leq C r^{-k} \int_{B_r(x)} |\theta_{4r}(y) - \theta_{r}(y)|d\mu(y).
     \end{equation}
\end{thm}

Given the $L^2$ Subspace Approximation, the proof of Theorem~\ref{A:thm:annular-structure} follows exactly the argument given in \cite[Chapter 8.2]{ChowJiangNaber},   applying the $W^{1,p}$ and subspace Reifenberg theorem of \cite{NaberValtorta2017}, \cite[Chapter 6]{ChowJiangNaber}. 
Indeed, the rest of the proof depends only on the properties of the annular region and the summing of the monotone quantity $\theta_r(x)$ which are available for stationary varifolds as well.

\subsection{Annular Decomposition Theorem}
We first recall the quantitative stratification.
\begin{defn}
    Let $M$ be a stationary integral varifold in $B_2$. We define 
    \begin{equation}
    \begin{gathered}
            \sS_{\eps,r}^k (M) = \left\{ x\in B_1\ssep \text{ for all }r\leq s<1, M\text{ is not $(k+1,\eps)$-symmetric on }B_s(x)  \right\},        \\
           \sS_\eps^k (M) = \left\{ x\in B_1\ssep \text{ for all }0<s<1, M\text{ is not $(k+1,\eps)$-symmetric on }B_s(x)  \right\}.
    \end{gathered}
    \end{equation}
    and note that $\sS_\eps^k =\sS_{\eps,0}^k = \bigcap_{r>0}\sS_{\eps,r}^k$.
\end{defn}

We  now state the annular decomposition theorem for all $k$-strata. We can in fact obtain a stronger statement than Theorem~\ref{thm:annular-decomposition}, which is obtained for harmonic maps in \cite{ChowJiangNaber}. To state it, we recall the definition of packing content
\begin{equation}
    \sP_r^k(S) = \sup \left\{ \sum_i \ome_kr_i^k \ssep x_i\in S \text{ with } \{B_{r_i}(x_i)\} \text{ disjoint, and $r_i\leq r$ } \right\}
\end{equation}
for any subset $S\subset \R^{n+m}$.

\begin{thm}[Annular Decomposition] \label{A:thm:annular-decomp} 
    Let $M$ be a stationary integral varifold in $B_{64}\subset \R^{n+m}$ with $|M|(B_{64})\leq \Lam$. Then for each $\eps>0$ and $\delta\leq \delta(n,m,\Lam,\eps)$ there exists $C=C(n,m,\Lam,\eps,\delta)>0$ and $\eps'=\eps'(n,m,\Lam,\eps)$ such that we can write 
    \begin{equation}
    \begin{gathered}
         B_1(0) \subset \sS^k_\eps \cup \bigcup_{b} B_{r_b}(x_b)\cup \bigcup_a \left( \Ann_a \cap B_{r_a}(x_a) \right) \\
         \sS_\eps^k\cap B_1(0)\subset  \sT^k\cup \tilde{\sS}_\eps^k:= \bigcup_a \left( T_a \cap B_{r_a}(x_a)\right) \cup \tilde{\sS}_\eps^k,
    \end{gathered}
    \end{equation}
    where

    \begin{itemize}
        \item [(1)] $\cA_a\subset B_{16r_a}(x_a)$ is a $(k,\eps',\delta)$-annular region with approximate singular set $T_a$.
        \item [(2)] Each $B_{16r_b}(x_b)$ is $(k+1,\eps)$-symmetric.
        \item [(3)] We have content bounds  $\sum_a r^k_a + \sum_b r_b^k\leq C$.
        \item [(4)] $\sT^k = \bigcup_a\left(T_a\cap B_{r_a}(x_a) \right)$ is a union of $k$-dimensional topological submanifolds, with the packing estimate $\sP_1^k(\sT^k)\leq C$.
        \item [(5)] For the residual singular set $\td \sS^k_\eps$, we have the packing estimate $\sP_1^k(\td \sS^k_\eps)\leq C$ and the measure vanishes: $\cH^k(\td \sS_\eps^k)=0$.
        \item [(6)] For every $0\leq r<1$ we have that $\sS_{\eps,r}^k\subset B_{Cr}\left(\sT^k\cup \td\sS_\eps^k\right)$.
    \end{itemize}
\end{thm}

An inspection of the proof of the Annular Decomposition Theorem for harmonic maps in \cite[Chapter 9]{ChowJiangNaber} shows that the proof of Theorem~\ref{A:thm:annular-decomp} depends only on several properties of the regularized energy $\theta_r(x)$ and its relationship to $(k,\delta)$-symmetry in \cite[Chapter 4]{ChowJiangNaber}, as well as the Annular Structure Theorem.
Once these facts are established, the proof in \cite[Chapter 9]{ChowJiangNaber} applies word-for-word replacing ``harmonic map" for ``varifold" and  replacing each of the cited statements in \cite[Chapter 4]{ChowJiangNaber} by the corresponding statements below.

Let us now prove each of the properties of $\theta_r(x)$ for stationary varifolds which are invoked in \cite[Chapter 9]{ChowJiangNaber}.
We first note an immediate consequence of \ref{prop:monotonicity}, which replaces \cite[Corollary 4.5]{ChowJiangNaber}
\begin{cor}
    Let $M$ be a stationary integer rectifiable varifold in $B_2$. If $B_{2r}(x)\subset B_2$ satisfies $|\theta_{2r}(x)-\theta_r(x)|\leq \delta$, then $M$ is $(0,C(n)\delta)$-symmetric on $B_r(x)$.
\end{cor}

We next study the behavior of $\theta_r(x)$ as $x$ varies. The following proposition replaces \cite[Theorem 4.17]{ChowJiangNaber} and \cite[Corollary 4.20]{ChowJiangNaber}. 
For convenience let us set $\vphi(x) = (1-|x|^2)_+$ and $\vphi_r(x) = \vphi\left(\frac{x}{r}\right)$ so that $\theta_r(x) = \frac{1}{\ome_n r^n}  \int \vphi_r(y-x) d|M|(y)$. Note that the argument below allows us to replace $\vphi$ by other kernels in the definition of regularized density.
\begin{prop}[Spatial Variations]
    Let $M$ be a stationary integer rectifiable varifold with $|M|(B_2)\leq \Lam$. Then the following hold:
    \begin{enumerate}
        \item [(1)] $\theta_r(x)$ is Lipschitz in $x$ and satisfies for every unit vector $v\in \R^{n+m}$,
        \begin{equation}
           | \ptl_v \theta_r(x) | \leq \frac{C(n)\Lam^{1/2}}{r}\left(  \frac{1}{\ome_n r^n}\int_{B_r(x)} \left| \pi_{T_yM}^\perp(v)\right|^2 d|M|(y)\right)^{1/2}.
        \end{equation}

        \item [(2)] For every \(\epsilon>0\) and \(\rho\in(0,1]\), there exists
    \( \delta=\delta(n,m,\Lambda,\epsilon,\rho)>0\)
    such that the following holds. If a \(k\)-dimensional linear subspace \(L^k\) satisfies
    \begin{equation}\frac{1}{\omega_n2^n}\int_{B_2}\left|\pi_{T_zM}^{\perp}|_{L^k}\right|^2\,d|M|(z)<\delta,\end{equation}
    then
    \( |\theta_r(x)-\theta_r(y)|<\epsilon\)
    for every \(r\in[\rho,1]\) and every \(x,y\in B_1\) with \(x-y\in L^k\).
    \end{enumerate}
\end{prop}
\begin{proof}
    Testing stationarity with the vector field $\vphi_r(y-x) v$, we obtain by direct differentiation
    \begin{equation}
        \begin{aligned}
            \partial_v \theta_r(x) = -\frac{1}{\ome_n r^{n+1}}\int \left\langle \pi_{T_yM}^\perp(v), \nabla \vphi \left( \frac{y-x}{r}\right) \right\rangle d|M|(y).
        \end{aligned}
    \end{equation}
    since $\int \langle \pi_{T_yM}(v), \nabla \vphi_r(y-x)\rangle  d|M| = 0$. By Cauchy-Schwarz and monotonicity we obtain (1).
    For (2), we let $z_t = (1-t)x + ty$ be the segment from $x$ to $y$. At every point along this segment we have, setting $v=\frac{y-x}{|y-x|}$, 
    \begin{equation}
        \frac{1}{\ome_n r^n}\int_{B_r(z_t)} \left| \pi_{T_wM}^\perp(v)\right|^2 d|M|(w) \leq 2^n r^{-n} \delta 
    \end{equation}
    and so by integration along $z_t$ we obtain
    \begin{equation}
        \begin{aligned}
            |\theta_r(x) -\theta_r(y)|\leq  C(n) |x-y| \Lam^{1/2} \delta^{1/2} r^{-1-n/2} \leq 2 C(n)\Lam^{1/2}\delta^{1/2} \rho^{-1-n/2} 
        \end{aligned}
    \end{equation}
    Choosing $\delta$ so that $ 2 C(n)\Lam^{1/2}\delta^{1/2} \rho^{-1-n/2}<\eps$ finishes the proof.
\end{proof}

The next proposition is a form of quantitative differentiation which establishes the existence of many scales at which $M$ is quantitatively symmetric \cite[Theorem 4.22, Corollary 4.24, Corollary 4.27]{ChowJiangNaber}.
\begin{prop}[Quantitative Differentiation] \label{A:prop:quantitative-diff}
    Let $M$ be a stationary integral varifold in $B_2$ with $|M|(B_2)\leq \Lam$ and let $\eps>0$.
    \begin{enumerate}
        \item [(1)] If we define for all $x\in B_1$ the set of $\eps$-bad scales
        \begin{equation}
            \mc B(x,\eps) = \{\alpha \in \N \ssep \text{ $M$ not $(0,\eps)$-symmetric on $B_{2^{-\alpha}}(x)$ } \}
        \end{equation}
        we have the cardinality estimate $|\mc B(x,\eps)|\leq C(n)\Lam \eps^{-1}$.

        \item [(2)] There exists $r_\eps = r(n,m,\Lam,\eps)$ such that for all $x\in B_1$ there exists $r=r(x)\geq r_\eps$ such that $M$ is $(0,\eps)$-symmetric  on $B_r(x)$.

        \item [(3)] If there is a $k$-dimensional subspace $L^k$ such that 
        \begin{equation}
            \frac{1}{\ome_n2^n} \int_{B_2} |\pi^\perp_{T_yM}|_{L^k}|^2 d|M|(y) < \delta(n,m,\Lam,\eps),
        \end{equation}
        then for every $x\in B_1$ there is $r=r(x)\geq r_{\eps}(n,m,\Lam,\eps)$ such that $M$ is $(k,\eps)$-symmetric on $B_r(x)$ with respect to $L^k$. 
    \end{enumerate}
\end{prop}
\begin{proof}
        At every $\eps$-bad scale $\alpha$ the monotonicity formula Proposition~\ref{prop:monotonicity} gives 
        \[
        \theta_{2^{-\alpha+1}}(x) - \theta_{2^{-\alpha}}(x) \geq c(n)\eps.
        \]
        Thus,
        \begin{equation}
            \begin{aligned}
                 c(n) \eps |\mc B(x,\eps)|&\leq \sum_{\alpha \in \mc B(x,\eps)} \left(\theta_{2^{-\alpha+1}}(x) - \theta_{2^{-\alpha}}(x)\right) \leq \sum_{\alpha \in \N}\left( \theta_{2^{-\alpha+1}}(x) - \theta_{2^{-\alpha}}(x)\right)\\
                 &\leq \theta_1(x) - \theta_0(x) \leq C \Lam
            \end{aligned}
        \end{equation}
        which proves (1). (2) follows immediately by observing that there are at most $C(n)\Lam \eps^{-1}$ bad scales and so if we consider the dyadic scales $1,\frac{1}{2},\ldots, 2^{-\lfloor C(n)\Lam \eps^{-1} \rfloor -1}$ at least one of them must be $(0,\eps)$-symmetric. 
        Finally to prove (3) we apply (2) with $\eps/4$ and let $r_\eps$ be the corresponding radius. For every $y\in B_r(x)$,
        \[
        |\pi_{T_yM}^\perp \pi_{L^k}^\perp(y-x)|^2 \leq 2|\pi_{T_yM}^\perp (y-x)|^2 + 2r^2 |\pi_{T_yM}^\perp |_{L^k} |^2 
        \]
        and so 
        \begin{equation*}
            \begin{aligned}
                \frac{1}{\ome_nr^n} &\int_{B_r(x)} \left( |\pi_{T_yM}^\perp|_{L^k}|^2 + r^{-2} |\pi_{T_yM}^\perp \pi_{L^k}^\perp (y-x)|^2\right) d|M|(y) \\
                &\leq \frac{3}{\ome_n r^n} \int_{B_r(x)} |\pi_{T_yM}^\perp|_{L^k} |^2 d|M|(y) + \frac{2}{\ome_n r^{n+2}}\int_{B_r(x)} |\pi_{T_yM}^\perp (y-x)|^2 d|M|(y) \\
                &\leq 3 r_\eps^{-n} 2^{n} \delta + \frac{\eps}{2}.
            \end{aligned} 
        \end{equation*}
        Choosing $\delta\leq (\eps/6)(r_\eps/2)^n$ we obtain the result.
\end{proof}

Next, we show a form of quantitative dimension reduction which corresponds to \cite[Theorem 4.31]{ChowJiangNaber}.
\begin{prop}[Quantitative Dimension Reduction]
    Let \(M\) be a stationary integral \(n\)-varifold in \(B_2\subset\mathbb R^{n+m}\), with
    \( |M|(B_2)\leq\Lambda.\)
    For every \(\epsilon\in(0,1)\), there exist constants
    \( \delta=\delta(n,m,\Lambda,\epsilon)>0, r_\epsilon=r_\epsilon(n,m,\Lambda,\epsilon)\in(0,1)\)
    such that the following holds.
    If \(M\) is \((k,\delta)\)-symmetric on \(B_2\) with respect to \(L^k\) and  \(x\in B_1\) satisfies
    \( \operatorname{dist}(x,L^k)>\epsilon,\)
    then there exists a radius
    \(\displaystyle r=r(x)\in[r_\epsilon,1]\)
    such that \(M\) is \((k+1,\epsilon)\)-symmetric on \(B_r(x)\).
\end{prop}
\begin{proof}
    Fix such an $x$  and let $e_{k+1} = \frac{\pi_{L^k}^\perp(x)}{|\pi_{L^k}^\perp (x)|}$ be the constant unit vector field obtained by projecting $x\in B_1\setminus \ol B_\eps(L^k)$ onto $L^\perp$ and normalizing.  
    Set $L^{k+1} = L^k \oplus \Span\{e_{k+1}\}$. For any $0<s\leq 1/4$ and $y\in B_{2s}(x)$ we can subtract $\pi_{L^k}^\perp(y-x)$ from $\pi^\perp_{L^k} y$ and use $|\pi_{L^k}^\perp x|>\eps$ to obtain,
    \[
    |\pi_{T_yM}^\perp e_{k+1}|^2\leq \frac{2}{\eps^2} |\pi_{T_yM}^\perp \pi_{L^k}^\perp y|^2 + \frac{8s^2}{\eps^2}.
    \]
    We then have 
    \begin{equation} \label{eq:A:quant-dimension-reduct-step}
        \begin{aligned}
           & \frac{1}{\ome_n(2s)^n} \int_{B_{2s}(x)} |\pi_{T_yM}^\perp |_{L^{k+1}}|^2 d|M|(y) \\
           &= \frac{1}{\ome_n(2s)^n} \int_{B_{2s}(x)} \left(|\pi_{T_yM}^\perp |_{L^{k}}|^2  + |\pi_{T_yM}^\perp e_{k+1}|^2 \right)d|M| (y)\\
            &\leq C(n) \eps^{-2} s^{-n}\delta + C(n) \Lam \eps^{-2} s^2.
        \end{aligned}
    \end{equation}
    We choose our constants in the following order. 
    Consider Proposition~\ref{A:prop:quantitative-diff}(3) with $L^{k+1}$ and $\eps$ and let $\delta'=\delta'(n,m,\Lam,\eps)$ and $r'=r'(n,m,\Lam,\eps)$ be the parameters obtained there. 
    We  choose $s=s(n,m,\Lam,\eps)\in (0,1/4]$ small enough so that $C(n)\Lam \eps^{-2} s^2 <\delta'/2$, and then choose $\delta>0$ small enough so that $C(n)\eps^{-2} s^{-n}\delta< \delta'/2$. By \ref{eq:A:quant-dimension-reduct-step} we may then apply Proposition~\ref{A:prop:quantitative-diff} at scale $s$ to obtain a radius $r\in [r's,s]$ such that $M$ is $(k+1,\eps)$-symmetric with respect to $L^{k+1}$ on $B_r(x)$, which finishes the proof.
\end{proof}

The quantitative cone splitting for varifolds is directly analogous to \cite[Theorem 4.39, Theorem 4.41]{ChowJiangNaber}
\begin{prop}[Quantitative Cone Splitting]
    Let \(M\) be an integral \(n\)-varifold in \(B_6\). Let \(0<\tau\leq1\), and suppose the points
\(\displaystyle x_0=0, x_1,\ldots,x_k\in B_1\)
are in \(\tau\)-general position. Let $L^k=\Span\{x_1-x_0,\ldots,x_k-x_0\}$.
\begin{enumerate}
    \item If \(M\) is \((0,\epsilon)\)-symmetric on each ball \(B_3(x_i)\),
    then  \(M\) is \((k,C(n,m,\tau)\epsilon)\)-symmetric on \(B_1(x_0)\) with respect to $L^k$.
    \item If $M$ is stationary and $\theta_5(x_i)-\theta_4(x_i)\leq\epsilon$ for all $i=0,\ldots, k$, then \(M\) is \((k,C(n,m,\tau)\epsilon)\)-symmetric on \(B_1(x_0)\), with respect to $L^k$.
    \end{enumerate}
        One can take $C(n,m,\tau) = C(n,m)\tau^{-2k}$ or the common constant $C(n,m) \tau^{-2n}$ for all $k\leq n$.
\end{prop}
\begin{proof}
          Subtracting the vectors $y-x_i$ and $y-x_0$ we immediately obtain
          \begin{equation*}
          \begin{aligned}
               \int_{B_1(x_0)}& |\pi_{T_yM}^\perp (x_i-x_0)|^2 d|M|(y) \\
               &\leq 2 \int_{B_1(x_0)}|\pi_{T_yM}^\perp (y-x_i)|^2 d|M|(y) + 2 \int_{B_1(x_0)}|\pi_{T_yM}^\perp (y-x_0)|^2 d|M|(y) < C(n) \eps.
          \end{aligned}
          \end{equation*}
          We can apply Gram-Schmidt to $x_1-x_0,\ldots, x_k-x_0$ to obtain an orthonormal basis $e_1,\ldots, e_k$ of $L^k$. 
          Since the $x_i$ are in $\tau$-general position, from the Gram-Schmidt process it is immediate to see that
          \begin{equation*}
              |\pi_{T_yM}^\perp e_j|^2 \leq 2\tau^{-2} \left( |\pi_{T_yM}^\perp(x_j-x_0)|^2 + 4 \sum_{i<j} |\pi_{T_yM}^\perp e_i|^2\right).
          \end{equation*}
          Hence by induction,
          \begin{equation*}
              \sum_{j=1}^k |\pi_{T_yM}^\perp e_j|^2 \leq C(n,m) \tau^{-2k} \sum_{j=1}^k |\pi_{T_yM}^\perp (x_j-x_0)|^2
          \end{equation*}
          so  by integrating we obtain $\int_{B_1(x_0)}|\pi_{T_yM}^\perp|_{L^k}|^2 d|M|(y)\leq C(n,m)\tau^{-2k} \eps$. 
          For the other term in the definition of symmetry, we write $\pi_{L^k}^\perp(y-x_0) = (y-x_0) - \pi_{L^k}(y-x_0)$. Since $|y-x_0|<1$ in $B_1(x_0)$, we have 
          \begin{equation*}
              |\pi_{T_yM}^\perp \pi_{L^k}^\perp(y-x_0)|^2 \leq 2|\pi_{T_yM}^\perp (y-x_0)|^2 + 2 |\pi_{T_yM}^\perp|_{L^k}|^2.
          \end{equation*}
          Integrating we obtain (1), and the monotonicity formula immediately yields (2).
\end{proof}

Finally, we establish the corresponding statement of \cite[Theorem 4.44]{ChowJiangNaber}, which says that very 0-symmetric balls allow to pass $k$-symmetries between scales. We first need a short lemma corresponding to \cite[Corollary 4.14]{ChowJiangNaber}.
\begin{lem}\label{A:lem:radial-comparison}
    Let \(M\) be a stationary integral \(n\)-varifold in \(B_4\) with \(|M|(B_4)\leq\Lambda\), and suppose \(M\) is \((0,\delta)\)-symmetric on \(B_4\). Then for \(x\in B_1\), \(1\leq a\leq2\), and \(0<s,t\leq1\),
    \begin{equation*}
        \left|\theta_{as}(sx)-\theta_{at}(tx)\right|\leq C(n)\sqrt{\Lambda\delta}\left|s^{-(n+2)/2}-t^{-(n+2)/2}\right|
    \end{equation*}
\end{lem}
\begin{proof}
    For $0<s\leq 1$, we differentiate the function $\theta_{as}(sx)$ and test stationarity with the vector field $X(y) = \left( 1- \frac{|y-sx|^2}{a^2 s^2}\right)_+ y$ to obtain
    \begin{equation*}
        \frac{d}{ds}\theta_{as}(sx) = \frac{2}{\ome_n a^{n+2} s^{n+3}} \int_{B_{as}(sx)} \langle \pi_{T_yM}^\perp y, y-sx\rangle d|M|(y).
    \end{equation*}
    By Cauchy-Schwarz, symmetry, and monotonicity we obtain
    \begin{equation*}
        \left|\frac{d}{ds}\theta_{as}(sx)\right|\leq C(n) s^{-n-2} \left( \int_{B_4} |\pi_{T_yM}^\perp y|^2 d|M|(y)\right)^{1/2} |M|(B_{3s})^{1/2} \leq C(n) \sqrt{\Lam \delta} s^{-(n+4)/2}.
    \end{equation*}
    Integrating in $s$ we prove the lemma.
\end{proof}

\begin{prop}
Let $M$ be a stationary integral varifold in $B_4$ with $|M|(B_4)\leq \Lam$ and suppose $M$ is $(0,\delta)$-symmetric on $B_4$. Then, for $1\geq r\geq r(\delta;n,m,\Lam,\eps)$ we have:
\begin{enumerate}
    \item [(1)] If $M$ is $(k,\eps)$-symmetric on $B_1$, then $M$ is $(k,C(n,m)\eps)$-symmetric on $B_r$.
    \item [(2)]  If $M$ is $(k,\eps)$-symmetric on $B_r$ then $M$ is $(k,C(n,m)\eps)$-symmetric on $B_1$.
\end{enumerate}
    Here $r(\delta;n,m,\Lam,\eps)$ means the constant tends to zero as $\delta\to 0$ as the other constants are held fixed.
\end{prop}
\begin{proof}
    Let us prove both items simultaneously. For $r(\delta;n,m,\Lam,\eps)$ to be chosen at the end, we 
    fix $s,t\in [r(\delta;n,m,\Lam,\eps),1]$ and 
    suppose $M$ is $(k,\eps)$-symmetric with respect to $L^k$ on $B_s$. We show that $M$ is $(k,C\eps)$-symmetric on $B_t$, and  take $(s,t) = (1,r)$ to prove (1) and $(s,t) = (r,1)$ to prove (2). 
    Letting $e_1,\ldots, e_k$ be an orthonormal basis for $L^k$, we set $x_0 = 0$ and $x_j = e_j/8$. 
    The $(k,\eps)$-symmetry implies 
    \begin{equation*}
        |\theta_{s/2}(sx_j/4) - \theta_{s/4}(sx_j/4)|\leq C(n)\eps
    \end{equation*}
    and so by Lemma~\ref{A:lem:radial-comparison} we can choose $\delta=\delta(n,m,\Lam,\eps, r)$ sufficiently small so that 
    \begin{equation*}
        |\theta_{at}(tx_j) - \theta_{as/4}(sx_j/4)|<\eps
    \end{equation*}
    for all $a\in [1,2]$. Note that $\delta$ and $r(\delta;n,m,\Lam,\eps)$ must be chosen so that $\delta\leq c(n,m) \frac{\eps^2}{1+\Lam} r(\delta;n,m,\Lam,\eps)^{n+2}$, which  shows the dependency on $\delta$. Summing the  estimates for $a=1,2$ gives
    \begin{equation*}
        \theta_{2t}(tx_j) - \theta_t(tx_j)\leq C(n)\eps.
    \end{equation*}
    From this pinching it is now easy to see the $(k,C(n,m)\eps)$-symmetry on $B_t$.
\end{proof}

\subsection*{Acknowledgments} 
The author is incredibly grateful to his advisor, Camillo De Lellis, for introducing the problem to him and for his amazing mentorship, advice, and support. He would also like to thank Aaron Naber and Zhenhua Liu for very insightful discussions and for their encouragement. 

This material is based upon work supported by the National Science Foundation Graduate Research Fellowship Program under Grant No. DGE-2444107. Any opinions, findings, and conclusions or recommendations expressed in this material are those of the author(s) and do not necessarily reflect the views of the National Science Foundation.

\subsection*{AI Disclosure}
The author acknowledges the use of AI tools, especially Codex with Chat-GPT Sol and Astra. 
Early on, the author had proved the main theorem in 2 dimensions with Gauss-Bonnet and it was clear that a more robust argument by a density pinching or superconvexity estimate would extend to  higher dimensions. After many attempts (both individually and with interaction with the LLM), Chat-GPT Sol 5.6 suggested the 2-dimensional Jacobi field argument sketched in the introduction, after which the rest of the argument was clear to the author. 
LLMs were also used to help proofread the article and  fill in details. 
All text in the article was written by the author, and he takes full responsibility for the content.

\bibliographystyle{alphaurl}
\bibliography{bibliography}

\end{document}